\documentclass[reqno, 12pt,a4letter]{amsart}
\usepackage{amsmath,amsxtra,amssymb,latexsym, amscd,amsthm}
\usepackage{graphicx,color}
\usepackage[utf8]{inputenc}
\usepackage[mathscr]{euscript}
\usepackage{mathrsfs}
\usepackage[english]{babel}
\usepackage{enumerate}
\usepackage[english]{babel}
\usepackage[active]{srcltx}
\usepackage{color}
\usepackage{cite}

\newtheorem {theorem}{Theorem}[section]
\newtheorem {corollary}{Corollary}[section]
\newtheorem {lemma}{Lemma}[section]
\newtheorem {definition}{Definition}[section]
\newtheorem {remark}{Remark}[section]

\def\ar{a\kern-.370em\raise.16ex\hbox{\char95\kern-0.53ex\char'47}\kern.05em}
\def\ees{{\accent"5E e}\kern-.385em\raise.2ex\hbox{\char'23}\kern-.08em}
\def\eex{{\accent"5E e}\kern-.470em\raise.3ex\hbox{\char'176}}
\def\AR{A\kern-.46em\raise.80ex\hbox{\char95\kern-0.53ex\char'47}\kern.13em}
\def\EES{{\accent"5E E}\kern-.5em\raise.8ex\hbox{\char'23 }}
\def\EEX{{\accent"5E E}\kern-.60em\raise.9ex\hbox{\char'176}\kern.1em}
\def\ow{o\kern-.42em\raise.82ex\hbox{
  \vrule width .12em height .0ex depth .075ex \kern-0.16em \char'56}\kern-.07em}
\def\OW{O\kern-.460em\raise1.36ex\hbox{
\vrule width .13em height .0ex depth .075ex \kern-0.16em \char'56}\kern-.07em}
\def\UW{U\kern-.42em\raise1.36ex\hbox{
\vrule width .13em height .0ex depth .075ex \kern-0.16em \char'56}\kern-.07em}
\def\DD{D\kern-.7em\raise0.4ex\hbox{\char '55}\kern.33em}
\title{On definable open continuous mappings}

\author{S\~i Ti\d{\^e}p \DD inh$^\dag$}
\address{Institute of Mathematics, VAST, 18, Hoang Quoc Viet Road, Cau Giay District 10307, Hanoi, Vietnam}
\email{dstiep@math.ac.vn}

\author{TI\EES N-S\OW N PH\d{A}M$^\ddag$}
\address{Department of Mathematics, Dalat University, 1 Phu Dong Thien Vuong, Dalat, Vietnam}
\email{sonpt@dlu.edu.vn}

\date{\today}
\subjclass[2010]{26B10, 54C10, 03C64}
\keywords{Brouwer degree, Jacobian, Whyburn's conjecture, definable mappings/sets, open mappings, homeomorphisms}

\begin{document}

\begin{abstract}
For a definable continuous mapping $f$ from a definable connected open subset $\Omega$ of $\mathbb R^n$ into $\mathbb R^n,$ we show that the following statements are equivalent:
\begin{enumerate}[{\rm (i)}]
\item The mapping $f$ is open.
\item The fibers of $f$ are finite and the Jacobian of $f$ does not change sign on the set of points at which $f$ is differentiable.
\item The fibers of ${f}$ are finite and the set of points at which $f$ is not a local homeomorphism has dimension at most $n - 2.$
\end{enumerate}

As an application, we prove that Whyburn's conjecture is true for definable mappings: A definable open continuous mapping of one closed ball into another which maps boundary homeomorphically onto boundary is necessarily a homeomorphism.
\end{abstract}

\maketitle

\section{Introduction}
We are interested in the possible necessary and sufficient conditions for a continuous mapping to be open. 
Recall that a mapping is {\em open} if it maps open sets onto open sets.

By {\em Remmert's open mapping theorem} (see, for example, \cite[Theorem~2, p.~297]{Lojasiewicz1991} or \cite[Proposition~4, p.~132]{Narasimhan1966}), it is well known that a holomorphic mapping $f\colon\mathbb C^n\to\mathbb C^n$ is open if and only if its fibers are discrete.

From the work of Church~\cite{Church1963} (see also \cite{Church1978}) we know that if a mapping $f \colon\mathbb R^n \to \mathbb R^n$ is $C^n$ and light (i.e., the fibers of $f$ are totally disconnected), then $f$ is open if and only if the Jacobian (i.e., the determinant of the Jacobian matrix) of $f$ does not change sign, equivalently, the set of points at which $f$ is not a local homeomorphism has dimension at most $n - 2.$

For a real polynomial mapping $f\colon\mathbb R^n\to\mathbb R^n$, Gamboa and Ronga~\cite{Gamboa1996} proved that $f$ is open if and only if the fibers of $f$ are finite and the Jacobian of $f$ does not change sign.
After that, Hirsch~\cite{Hirsch2002} affirmed that the Jacobian of a real analytic open mapping $f\colon\mathbb R^n\to\mathbb R^n$ does not
change sign. Recently in \cite[Theorem 3.14]{Denkowski2017}, Denkowski and Loeb showed that a subanalytic (or definable in some o-minimal structure) mapping $f$ of class $C^1$ is open if and only if the fibers of $f$ are discrete and the Jacobian of $f$ does not change sign. 

In the non-smooth setting, a result of Scholtes~\cite{Scholtes2012} (see Corollary~\ref{Corollary31} below) stated that a piecewise affine mapping 
$f \colon \mathbb{R}^n \to \mathbb{R}^n$ is open if and only if it is coherently oriented (meaning that the Jacobian of affine mappings associated to $f$ have the same nonzero sign). See also \cite{Borwein1988, Gowda1996, Lee2021, Mordukhovich2018, Penot1989, Yen2008} for related works.

We present here a {\em definable non-smooth} version of the above results. Namely, let ${f} \colon \Omega \to \mathbb{R}^n$ be a definable continuous mapping, where $\Omega$ is a definable connected open set in $\mathbb R^n.$ Denote by $D_f$ the set of points in $\Omega$ at which $f$ is differentiable; this set is dense in $\Omega$ (see, e.g., \cite{Coste2000, Dries1998}). We also denote by $B_f$ the set of points in $\Omega$ at which $f$ fails to be a local homeomorphism. Then the following statements are equivalent:
\begin{enumerate}[{\rm (i)}]
\item The mapping $f$ is open.
\item The fibers of $f$ are finite and the Jacobian of $f$ does not change sign on $D_f.$
\item The fibers of ${f}$ are finite and the set $B_f$ has dimension at most $n - 2.$
\end{enumerate}

The idea of the proof of the equivalence (i)~$\Leftrightarrow$~(ii) is similar to those in \cite{Denkowski2017, Gamboa1996, Hirsch2002}. However, more arguments need to be taken into account in the non-smooth case.

\medskip
In \cite{Whyburn1951} Whyburn stated the following conjecture and verified it for the case $n = 2$: {\em Suppose that $f$ is a light open continuous mapping of one closed ball of dimension $n$ onto another. If $f$ maps the boundary homeomorphically, then $f$ is a homeomorphism.}

Whyburn's conjecture has been proved for several class of continuous mappings by McAuley~\cite{McAuley1965}.
The conjecture has been verified for differentiable mappings with an additional hypothesis by Cronin and McAuley~\cite{Cronin1966}.
It is also valid for $C^2$ mappings due to Marx~\cite{Marx1968}.
On the other hand, it has been shown by Wilson~\cite{Wilson1973} that the conjecture is false for continuous mappings on higher dimensional spaces.
However, the second result of this paper shows that the conjecture is still true for definable continuous mappings.

\medskip
The rest of this paper is organized as follows. Section~\ref{SectionPreliminary} contains some properties of definable sets and mappings in o-minimal structures. For the convenience of the reader, the classical invariance of domain theorem and some properties of the Brouwer degree are also recalled here. 
Finally, results are given in Section~\ref{Section3}.

\section{Preliminaries} \label{SectionPreliminary}

\subsection{Notation and definitions}
We suppose $1 \le n \in \mathbb{N}$ and abbreviate $(x_1, \ldots, x_n)$ by $x.$ The space $\mathbb{R}^n$ is equipped with the usual scalar product $\langle \cdot, \cdot \rangle$ and the corresponding Euclidean norm $\| \cdot\|.$ The open ball and sphere of radius $r$ centered at the origin in $\mathbb{R}^n$ will be denoted by $\mathbb{B}^n_r$ and $\mathbb{S}^{n - 1}_r,$  respectively. If $x \in \mathbb{R}^n$ and $\Omega \subset \mathbb{R}^n$ is a non-empty set then the distance from $x$ to $\Omega$ is defined by $\mathrm{dist}(x, \Omega):=\inf_{y \in \Omega}\|x - y\|.$ 
The closure and boundary of a set $\Omega \subset \mathbb{R}^n$ will be written as $\overline{\Omega}$  and $\partial \Omega,$ respectively. 

Let ${f}$ be a mapping from an open set $\Omega$ in $\mathbb{R}^n$ into $\mathbb{R}^n.$ Then $f$ is an {\em open mapping} if $f(U)$ is an open subset of $\mathbb{R}^n$ whenever $U$ is an open subset of $\Omega;$ $f$ has {\em finite fibers} if for each $y \in \mathbb{R}^n,$ the fiber $f^{-1}(y)$ is a finite (possibly empty) set. Let $D_f$ denote the set of points in $\Omega$ at which $f$ is differentiable.
If $x \in D_f$ then we denote the Jacobian matrix of ${f}$ at $x$ by $d{f}(x) =\left[\frac{\partial {f}_i}{\partial x_j}(x)\right],$ and the determinant of $d{f}(x)$ is the {\em Jacobian} of ${f}$ at $x,$ denoted by $J{f}(x).$ 
Let $R_f$ denote the set of points $x \in \Omega$ such that ${f}$ is of class $C^1$ in a neighborhood of $x$ and the Jacobian $Jf(x)$ is nonzero. Observe that  $R_f$ is an open set  but it is not necessarily connected.  As in \cite{Church1963}, the {\em branch} set of $f,$ denoted by $B_f,$ is the set of points in $\Omega$ at which $f$ fails to be a local homeomorphism. Note that $B_f \subset \Omega \setminus R_f$ by the inverse mapping theorem.

\subsection{The invariance of domain theorem and the Brouwer degree}

For the convenience of the reader, we recall here the classical invariance of domain theorem and some properties of the Brouwer degree.

\begin{lemma} [Invariance of domain] \label{Brouwer1}
Let $\Omega$ be an open subset of $\mathbb{R}^n.$ Then every injective continuous mapping from $\Omega$ into $\mathbb{R}^n$ is open.
\end{lemma}
\begin{proof}
See, for example, \cite[Chapter~4, Proposition~7.4]{Dold1972}.
\end{proof}

Suppose $\Omega$ is an open and bounded set in $\mathbb{R}^n,$ $f \colon \overline{\Omega} \rightarrow \mathbb{R}^n$ is a continuous mapping, and $y \not \in f(\partial \Omega).$ The {\em Brouwer degree} of $f$ on $\Omega$ with respect to $y,$ denoted by $\deg(f, \Omega, y),$ is a function of integer values which enjoys several important properties (normalization, domain decomposition, local constancy, homotopy invariance, etc.). The Brouwer degree is a power tool used in analysis and topology, in particular, it gives an estimation and the nature of the solution(s) of the equation $f(x) = y$ in $\Omega.$ For more details, the reader may refer to \cite{Deimling1985, LLoyd1978} and the references therein. 

The two lemmas below provide some useful properties of the Brouwer degree.

\begin{lemma}\label{Brouwer2}
Let $\Omega$ be an open and bounded set in $\mathbb{R}^n,$ $f \colon \overline{\Omega} \rightarrow \mathbb{R}^n$ be a continuous mapping, and $y \not \in f(\partial \Omega).$  Then the following statements hold:
\begin{enumerate}[{\rm (i)}]
\item If $\deg(f, \Omega, y) \ne 0,$ then there is $x \in \Omega$ such that $f(x) = y.$

\item For all $y' \in \mathbb{R}^n$ with $\|y' - y\| < \mathrm{dist} (y, f(\partial \Omega))$ we have
\begin{eqnarray*}
\deg(f, \Omega, y') & = & \deg(f, \Omega, y).
\end{eqnarray*}

\item If $H \colon \overline{\Omega} \times [0,1] \to \mathbb{R}^n$ is continuous, $\gamma \colon  [0,1] \to \mathbb{R}^n$ is continuous, and $\gamma(t) \not \in H(\partial \Omega, t)$ for every $t\in [0,1]$, then $\deg(H(\cdot, t), \Omega, \gamma(t))$ is independent of $t\in [0,1]$.
\end{enumerate}
\end{lemma}
\begin{proof}
See, for example, \cite{Deimling1985, LLoyd1978}.
\end{proof}

\begin{lemma}\label{Brouwer3}
Let $\Omega$ be an open and bounded set in $\mathbb{R}^n$ and $f \colon \overline{\Omega} \rightarrow \mathbb{R}^n$ be a continuous mapping. Then the following statements hold:
\begin{enumerate}[{\rm (i)}]
\item Let $x \in \Omega$ be such that the Jacobian matrix $df(x)$ exists and is nonsingular. Then there exists a neighborhood $W$ of $x$ such that $\overline{W} \cap f^{-1}(f(x)) = \{x\}$ and 
\begin{eqnarray*}
\deg (f, W, f(x)) &=& \mathrm{sign} Jf(x).
\end{eqnarray*}

\item Let $y \not \in f(\partial \Omega)$ be such that for every $x \in f^{-1}(y),$ the Jacobian matrix $df(x)$ exists and is nonsingular. Then
\begin{eqnarray*}
\deg (f, \Omega, y) &=& \sum_{x \in f^{-1}(y)} \mathrm{sign} Jf(x).
\end{eqnarray*}
\end{enumerate}
\end{lemma}
\begin{proof}
See \cite{Deimling1985, LLoyd1978} for the smooth case and see \cite{HaTDX2021, Pourciau1983-2, Shannon1994} for the non-smooth case.
\end{proof}

\subsection{O-minimal structures and definable mappings}

The notion of o-minimality was developed in the late 1980s after it was noticed that many proofs of analytic and geometric properties of semi-algebraic sets and mappings can be carried over verbatim for subanalytic sets and mappings. We refer the reader to \cite{Coste2000, Dries1998, Dries1996} for the basic properties of o-minimal structures used in this paper. 

\begin{definition}{\rm
An {\em o-minimal structure} on $(\mathbb{R}, +, \cdot)$ is a sequence $\mathcal{D} := (\mathcal{D}_n)_{n \in \mathbb{N}}$ such that for each $n \in \mathbb{N}$:
\begin{itemize}
\item [(a)] $\mathcal{D}_n$ is a Boolean algebra of subsets of $\mathbb{R}^n$.
\item [(b)] If $X \in \mathcal{D}_m$ and $Y \in \mathcal{D}_n$, then $X \times Y \in \mathcal {D}_{m+n}.$
\item [(c)] If $X \in \mathcal {D}_{n + 1},$ then $\pi(X) \in \mathcal {D}_n,$ where $\pi \colon \mathbb{R}^{n+1} \to \mathbb{R}^n$ is the projection on the first $n$ coordinates.
\item [(d)] $\mathcal{D}_n$ contains all algebraic subsets of $\mathbb{R}^n.$
\item [(e)] Each set belonging to $\mathcal{D}_1$ is a finite union of points and intervals.
\end{itemize}
}\end{definition}

A set belonging to $\mathcal{D}$ is said to be {\em definable} (in that structure). {\em Definable mappings} in structure $\mathcal{D}$ are mappings whose graphs are definable sets in $\mathcal{D}.$

Examples of o-minimal structures are
\begin{itemize}
\item the semi-algebraic sets (by the Tarski--Seidenberg theorem),
\item the globally subanalytic sets, i.e., the subanalytic sets of $\mathbb{R}^n$ whose (compact) closures in the real projective space $\mathbb{R}\mathbb{P}^n$ are subanalytic (using Gabrielov's complement theorem).
\end{itemize}

In this note, we fix an arbitrary o-minimal structure on $(\mathbb{R}, +, \cdot).$ The term ``definable'' means definable in this structure.
We recall some useful facts which we shall need later.

\begin{lemma} [{Monotonicity}] \label{MonotonicityLemma}
Let $f \colon (a, b) \rightarrow \mathbb{R}$ be a definable function and $p$ be a positive integer. Then there are finitely many points $a = t_0 < t_1 < \cdots < t_k = b$ such that the restriction of $f$ to each interval $(t_i, t_{i + 1})$ is of class $C^p$, and either constant or strictly monotone.
\end{lemma}

\begin{lemma}[Path connectedness]\label{PathConnectedness}
The following statements hold:
\begin{enumerate}[{\rm (i)}]
\item Every definable set has a finite number of connected components and each such component is definable.

\item Every definable connected set $X$ is path connected, i.e., for every points $x, y$ in $X,$ there exists a definable continuous curve $\gamma \colon [0, 1] \to X$ such that $\gamma(0)  = x$ and $\gamma(1) = y.$
\end{enumerate}
\end{lemma}

By the cell decomposition theorem (see, for example~\cite[4.2]{Dries1996}), for any $p \in \mathbb{N}$ and any nonempty definable subset $X$ of $\mathbb{R}^n,$ we can write $X$ as a disjoint union of finitely many definable $C^p$-manifolds of different dimensions. The {\em dimension} $\dim X$ of a nonempty definable set $X$ can thus be defined as the dimension of the manifold of highest dimension of such a decomposition. This dimension is well defined and independent on the decomposition of $X.$ By convention, the dimension of the empty set is taken to be negative infinity. 
A point $x\in X$ is {\em generic} if there exists a neighborhood $U$ of $x$ in $\mathbb R^n$ such that $X\cap U$ is a definable $C^1$-manifold of dimension $\dim X.$

\begin{lemma}\label{DimensionLemma}
Let $X \subset \mathbb R^n$ be a nonempty definable set. Then the following statements hold:
\begin{enumerate}[{\rm (i)}]
\item The set $X$ has measure zero if and only if $\dim X <n.$
\item The interior of $X$ is nonempty if and only if $\dim X = n.$
\item $\dim (\overline{X}\setminus X) < \dim X.$ In particular,  $\dim \overline{X} = \dim X.$
\item Let $Y \subset \mathbb{R}^n$ be a definable set containing $X.$ If $X$ is dense in $Y$ then $\dim (Y \setminus X) < \dim Y.$
\item If $Y \subset \mathbb{R}^n$ is definable, then $\dim (X \cup Y) = \max \{\dim X, \dim Y\}.$
\item Let $f \colon X \rightarrow \mathbb R^m$ be a definable mapping. Then $\dim f(X) \leq \dim X.$ 
\item The complement in $X$ of the set of generic points in $X$ is a definable set of dimension less than $\dim X.$
\end{enumerate}
\end{lemma}

\begin{lemma}\label{DiffrentiableLemma}
Let $X \subset \mathbb{R}^n$ be a definable open set and $f \colon X \to \mathbb R^m$ be a definable mapping. Then 
for each positive integer $p,$ the set of points where $f$ is not of class $C^p$ is a definable set of dimension less than $n.$
\end{lemma}

In the sequel we will make use of Hardt's triviality (see \cite{Hardt1980, Dries1998}).

\begin{theorem}[Hardt's triviality] \label{HardtTheorem}
Consider a definable continuous mapping $f \colon X \rightarrow Y$ where $X \subset \mathbb{R}^n$ and $Y \subset \mathbb{R}^m$ are definable sets. Then there exists a finite partition $Y  = Y_1 \cup \cdots \cup Y_k$ of $Y$ into definable sets $Y_i$ such that $f$ is definably trivial over each $Y_i,$ that is, there exists a definable set $F_i \subset \mathbb{R}^{n_i},$ for some $n_i,$ and a definable homeomorphism $h_i \colon f^{-1} (Y_i) \rightarrow Y_i \times F_i$ such that the composition $h_i$ with the projection $Y_i \times F_i \rightarrow Y_i, (y, z) \mapsto y,$  is equal to the restriction of $f$ to $f^{-1} (Y_i).$
\end{theorem}

\section{Results and proofs} \label{Section3}

The following result provides necessary and sufficient conditions for a definable continuous mapping to be open.
For related results, we refer the reader to \cite{Chernavski1964, Chernavski1965, Church1960, Church1963, Church1967, Church1978, Denkowski2017, Gamboa1996, Hirsch2002, Titus1952, Vaisala1966}.
 
\begin{theorem}\label{MainTheorem} , 
Let ${f} \colon \Omega \to \mathbb{R}^n$ be a definable continuous mapping, where $\Omega$ is a definable connected open set in $\mathbb R^n.$
Then the following two conditions are equivalent:
\begin{enumerate}[{\rm (i)}]
\item The mapping ${f}$ is open.
\item The fibers of ${f}$ are finite and the Jacobian $Jf$ does not change sign on $D_{f}.$ 
\item The fibers of ${f}$ are finite and the Jacobian $Jf$ does not change sign on $R_{f}.$
\item The fibers of ${f}$ are finite and the branch set $B_f$ has dimension at most $n - 2.$
\end{enumerate}
\end{theorem}

In order to prove Theorem~\ref{MainTheorem}, we need some lemmas. 
The first one is \cite[Fact~2.2]{Peterzil2007}, which is a consequence of \cite[Proposition~2]{Johns2001}; we recall it here together with a direct proof.

\begin{lemma}  \label{Lemma31}
Let $\Omega$ be a definable connected open subset of $\mathbb{R}^n$ and $R$ be a definable  open dense subset of $\Omega.$ Consider the graph $\Gamma$ whose vertices are the connected components of $R$ with two components $R_i$ and $R_j$ connected by an edge if and only if 
$\dim (\Omega \cap \overline{R}_i\cap \overline{R}_j)\geqslant n - 1.$ Then  $\Gamma$ is connected.
\end{lemma}
\begin{proof} 
Suppose for contradiction that the graph $\Gamma$ is not connected. There exists a connected component of $\Gamma$ whose vertices are $R_1, \ldots, R_p$ for some $p \geqslant 1,$ and let $R_{p + 1}, \ldots, R_{q}$ with $q \geqslant p + 1$ be the remaining vertices of $\Gamma.$
By assumption, we have for all $i = 1, \ldots, p,$ and all $j = p + 1, \ldots, q,$
$$\dim(\Omega \cap \overline R_i\cap\overline R_j) < n - 1.$$

Set 
$$X_1:= \bigcup_{i=1}^p (\Omega \cap \overline R_i) \quad \text{ and } \quad  X_2 := \bigcup_{j=p+1}^q (\Omega \cap \overline R_j).$$
Observe that
$$X_1\cap X_2 = \left(\bigcup_{i=1}^p (\Omega \cap \overline R_i) \right)\bigcap\left(\bigcup_{j=p+1}^q (\Omega \cap  \overline R_j)\right)=\bigcup_{\substack{i=1,\dots,p\\ j=p + 1,\dots,q}} (\Omega \cap  \overline R_i\cap \overline R_j).$$
Hence $\dim (X_1\cap X_2) < n - 1$ and so $\Omega \setminus(X_1\cap X_2)$ is path connected. On the other hand, since $R = \cup_{i = 1, \ldots, q} R_i$ is dense in $\Omega,$ we have $\Omega = X_1 \cup X_2$ and so
$$\Omega\setminus(X_1\cap X_2) = (X_1 \setminus X_2) \cup (X_2 \setminus X_1).$$
Therefore, there is a continuous path
$$\gamma\colon[0,1]\to\Omega\setminus(X_1\cap X_2)$$ 
such that $\gamma(0)\in X_1 \setminus X_2$ and $\gamma(1)\in X_2 \setminus X_1.$ Set 
$$t_* :=\sup \{t \in[0, 1]:\ \gamma(s) \in X_1 \setminus X_2 \ \textrm{ for all } \ s \in [0, t) \}.$$
Then it is easy to check that $\gamma(t_*) \in X_1\cap X_2$, which is a contradiction.
The lemma is proved.
\end{proof}

The following result is taken from \cite[Lemma~3.1]{Peterzil2007}.

\begin{lemma} \label{Lemma32}
Let $W\subset \mathbb R^{n-1}$ and $U \subset\mathbb R^n$ be definable open sets such that the set $\widetilde W : = W \times \{0\}$ is contained in the boundary of $U.$ Let $f \colon U \cup \widetilde W\to \mathbb R$ be a definable function continuous on $U \cup \widetilde W$ and $C^1$ on $U.$
Let $g \colon W \to \mathbb R$ be the function $y \mapsto f(y, 0)$ and $a \in W$ be a generic point. 
Then $g$ is differentiable at $a$ and, for all $i = 1,\dots, n - 1,$
$$\frac{\partial g}{\partial x_i}(a)\ =\ \lim_{x \to (a, 0),\, x \in U}\frac{\partial f}{\partial x_i}(x).$$
\end{lemma}

The following fact is simple but useful.

\begin{lemma} \label{Lemma33}
Let ${f} \colon \Omega \to \mathbb{R}^n$ be a definable continuous mapping, where $\Omega$ is a definable open set in $\mathbb R^n.$
Assume that the fibers of $f$ are finite. Then the following two statements hold:
\begin{enumerate}[{\rm (i)}]
\item If $X$ is a definable subset of $\Omega,$ then $\dim X = \dim f(X).$
\item The set $R_f$ is dense in $\Omega.$
\end{enumerate}
\end{lemma}

\begin{proof}
(i) Let $X$ be a definable subset of $\Omega.$ Applying Hardt's Trivial Theorem~\ref{HardtTheorem} to the restriction mapping $f|_X$ of $f$ to $X,$
we obtain a finite definable partition of $f(X)$ onto $Y_1, \ldots, Y_k$ such that $f|_X$ is definably trivial over each $Y_i.$ Since the fibers of $f$ are finite, it follows that 
$$\dim (f|_X)^{-1}(Y_i)\ =\ \dim Y_i \quad \textrm{ for } \quad i = 1, \ldots, k.$$
Hence 
$$\dim X\ =\ \max_{i = 1, \ldots, k} \dim (f|_X)^{-1}(Y_i)\ =\ \max_{i = 1, \ldots, k} \dim Y_i\ =\ \dim f(X),$$
which yields (i).

(ii) Let $B$ be the set of points $x \in \Omega$ such that ${f}$ is not $C^1$ in a neighborhood of $x.$ Then $B$ is closed in $\Omega.$ Further,
in view of Lemma~\ref{DiffrentiableLemma}, $B$ is a definable set of dimension less than $n.$ Consider the definable set
\begin{eqnarray*}
{C} &:=& \{x \in \Omega \setminus B \ : \ J{f}(x) = 0 \}.
\end{eqnarray*}
Applying Sard's theorem to the $C^1$-mapping 
$$\Omega \setminus B \to \mathbb{R}^n, \quad x \mapsto {f}(x),$$ 
we have that ${f}({C})$ has measure zero, and so it has dimension less than $n$ in view of Lemma~\ref{DimensionLemma}(i). 
This, together with the statement~(i), implies that
$$\dim {C}\ =\ \dim f({C}) < n.$$
Consequently, $\dim (B \cup {C}) < n.$ 
Now (ii) follows immediately since $R_f  = \Omega \setminus (B \cup {C}).$
This ends the proof of the lemma.
\end{proof}

We are now in position to prove Theorem~\ref{MainTheorem}.

\begin{proof}[Proof of Theorem~\ref{MainTheorem}] (cf. \cite[Theorem~3.14]{Denkowski2017}, \cite[Theorem~3.2]{Peterzil2007}).

\medskip 
(ii) $\Rightarrow$ (iii): Obviously.

\medskip 
(iii) $\Rightarrow$ (i): 
To see that $f$ is open, take any $\overline{x} \in \Omega$ and let $W$ be a neighborhood of $\overline{x}.$ Since the fiber ${f}^{-1}({f}(\overline{x}))$ is finite, there exists an open ball $U$ centered at $\overline{x}$ with $\overline{U} \subset W$ such that $\overline{U} \cap {f}^{-1}({f}(\overline{x})) = \{\overline{x}\}.$ Since $\partial{U}$ is compact and ${f}(\overline{x}) \not \in {f}(\partial U),$ we can find an open ball $V \subset \mathbb{R}^n$ centered at ${f}(\overline{x})$ such that ${f}(\partial U) \cap V = \emptyset.$ 
Now the conclusion follows if we can show that ${f}(U) \supset V.$ 

Recall that $R_f$ denotes the set $R_f$ of points $x \in \Omega$ such that ${f}$ is of class $C^1$ on a neighborhood of $x$ and the Jacobian $Jf(x)$ is nonzero. In view of Lemma~\ref{Lemma33}(ii), $R_f$ is dense in $\Omega.$ Since $U$ is an open subset of $\Omega,$ 
there exists a point $u$ in $U \cap R_f.$ Then ${f}$ is of class $C^1$ on a neighborhood of $u$ and the Jacobian matrix of $f$ at $u$ is nonsingular. By the inverse mapping theorem, the image ${f}(U)$ must contain an open set $Y \subset V.$ 
On the other hand, it follows from Lemmas~\ref{Lemma33}~and~\ref{DimensionLemma}(iv) that
\begin{eqnarray*}
\dim f(\Omega \setminus R_f) &=& \dim(\Omega \setminus R_f) \ < \ n.
\end{eqnarray*}
So there must exist some point $\overline{y}$ in $Y$ with $f^{-1}(y) \subset R_f.$ 
Therefore, $U \cap f^{-1}(\overline{y})$ is nonempty and for every $x \in U \cap f^{-1}(\overline{y}),$ the Jacobian matrix $df(x)$ exists and is nonsingular. 
Furthermore, by construction, $\overline{y}$ is not in ${f}(\partial U).$ For simplicity of notation, we let ${g}$ stand for the restriction of $f$ to $\overline{U}.$ Then the Brouwer degree $\deg ({g}, U, \overline{y})$ is defined. In light of Lemma~\ref{Brouwer3}(ii), we have
\begin{eqnarray*}
\deg ({g}, U, \overline{y}) &=& \sum_{x \in {g}^{-1}(\overline{y})} \mathrm{sign} J{g}(x).
\end{eqnarray*}
On the other hand, by assumption, the Jacobian $Jf$ is positive (or negative) on $R_{f},$ and so
\begin{eqnarray*}
\sum_{x \in {g}^{-1}(\overline{y})} \mathrm{sign} J{g}(x) &\ne& 0.
\end{eqnarray*}
Therefore $\deg ({g}, U, \overline{y}) \ne 0.$ 

Finally, to prove that $V \subset {f}(U),$ take any $y \in V$ and consider the continuous curve 
$$\gamma \colon [0, 1] \to V, \quad t \mapsto (1 - t)y + t \overline{y},$$ 
connecting $y$ to $\overline{y}.$ Since ${f}(\partial U) \cap V = \emptyset,$ we have $\gamma(t) \not \in {g}(\partial U)$ for all $t \in [0, 1].$ This, together with Lemma~\ref{Brouwer2}(iii), implies that
$$\deg ({g}, U,  y) \ = \ \deg ({g}, U, \overline{y}) \ \ne \ 0.$$
By Lemma~\ref{Brouwer2}(i), ${g}(x) = y$ for some $x$ in $U,$ and this proves the openness of $f.$

\medskip 
(i) $\Rightarrow$ (ii):
Assume that ${f}$ is open. If $n = 1,$ then ${f}$ is strictly monotone and there is nothing to prove. So for the rest of the proof we assume that $n > 1.$ 

By \cite[Theorem~3.10]{Denkowski2017} (see also \cite[Proposition, page~298]{Gamboa1996}), the fibers of $f$ are finite.

We first show that the Jacobian $Jf$ has constant sign on $R_f$ which means that $Jf$ is positive (or negative) on $R_f.$ 

Obviously, the set $R_f$ is definable open, and according to Lemma~\ref{Lemma33}(ii), it is dense in $\Omega.$ By Lemma~\ref{PathConnectedness}, $R_f$ has a finite number of connected components, say ${R}_1, \ldots, {R}_k,$ and these components are path connected. By definition, $J{f}$ has constant sign on each ${R}_i.$
According to Lemma~\ref{Lemma31}, we need to show that for any two components ${R}_i, {R}_j$ with $\dim (\Omega \cap \overline{R}_i \cap \overline{R}_j) \geqslant n - 1,$ the sign of $J{f}$ on ${R}_i$ is the same as on ${R}_j.$ 
We consider two such components ${R}_i, {R}_j$ and assume that they are ${R}_1$ and ${R}_2.$ 
Let $\Sigma := \Omega \cap \overline{R}_1 \cap \overline{R}_2.$ In view of Lemma~\ref{DimensionLemma}(iii), it is easy to see that $\dim \Sigma = n - 1.$

Let $\overline{x}$ be a generic point in $\Sigma;$ then $\Sigma$ is a $C^1$-submanifold of $\mathbb{R}^n$ of dimension $n - 1$ near $\overline{x}.$ 
Hence, there is a definable connected open neighborhood $U \subset \Omega$ of $\overline{x}$ and a definable diffeomorphism $\Phi$ from $U$ onto an open subset of $\mathbb R^n$ such that $\Phi(U \cap \Sigma) \subset \{x_n = 0\}.$ 
Shrinking $U$ and composing $\Phi$ with the reflection with respect to the hyperplane $\{x_n = 0\}$ if necessary, we may assume that $\Phi(U \cap {R}_1) \subset \{x_n > 0\}$ and $\Phi(U \cap {R}_2) \subset \{x_n < 0\}.$

Since $\Phi$ is a diffeomorphism and $U$ is connected, the sign of $J\Phi$ is constant on $U.$ Furthermore ${f} \circ \Phi^{-1}$ is open on $\Phi(U).$ Hence we may replace $\Omega$ by $\Phi(U),$ ${f}$ by ${f} \circ \Phi^{-1}$ and assume that
$${R}_1 \subset \{x_n > 0\}, \quad {R}_2 \subset \{x_n < 0\} \quad \text{ and } \quad \Sigma \subset \{x_n = 0\}.$$

On the other hand, since the fibers of $f$ are finite, it follows from Lemma~\ref{Lemma33}(i) that
\begin{eqnarray*}
\dim f(\Sigma)  &=& \dim \Sigma \ = \ n - 1.
\end{eqnarray*}
Let $\overline{y}$ be a generic point in $f(\Sigma);$ then $f(\Sigma)$ is a $C^1$-submanifold of $\mathbb{R}^n$ of dimension $n - 1$ near $\overline{y}.$ As before, by applying an appropriate definable diffeomorphism on an open neighborhood of $\overline{y}$ we may assume that ${f}(\Sigma) \subset \{y_n = 0\}.$ 
Applying Hardt's Trivial Theorem~\ref{HardtTheorem} to the restriction of $f$ to $f^{-1}(f(\Sigma)),$
we obtain a finite definable partition of $f(\Sigma)$ onto $Y_1, \ldots, Y_k$ such that this restriction mapping is definably trivial over each $Y_i.$ 
By Lemma~\ref{PathConnectedness}, $f^{-1}(Y_i)$ has a finite number of connected components and each such component is homeomorphic 
to $Y_i$ because $f$ has finite fibers. Observe that there exists an index $i$ such that the set  $f^{-1}(Y_i) \cap \Sigma$ is of dimension $n - 1.$ Let $\widetilde{\Sigma}$ be a connected component of  $f^{-1}(Y_i) \cap \Sigma$ of dimension $n - 1.$ Now, by shrinking $\Omega$ we may assume that $\Omega \cap f^{-1}(f(\widetilde{\Sigma})) = \widetilde{\Sigma}.$

By construction, it is easy to see that each of the (connected open) sets $f(R_1)$ and $f(R_2)$ is contained in either $\{y_n > 0\}$ or $\{y_n < 0\}$ but not in both. Furthermore, since $f$ is open, the sets $f(R_1)$ and $f(R_2)$ cannot lie in the same half-space $\{y_n > 0\}$ or $\{y_n < 0\}.$ Hence, without lost of generality, we may assume that
$${f}({R}_1) \subset \{y_n > 0\} \quad \textrm{ and } \quad {f}({R}_2) \subset \{y_n < 0\}.$$

Write ${f} := ({f}_1, \dots, {f}_n)$ and let $(a, 0) \in \mathbb{R}^{n - 1} \times \{0\}$ be a generic point in $\widetilde{\Sigma}.$ Observe that
$$f_n(a, -t)\  <\ f_n(a, 0)\ =\ 0\ <\ f_n(a, t)$$
for all $t > 0$ sufficiently small. By Lemma~\ref{MonotonicityLemma}, it is easy to see that ${f}_n(a, t)$ is strictly increasing in $t$ near $0.$ Consequently, we can find $\epsilon > 0$ small enough such that $\frac{\partial {f}_n}{\partial x_n}(a, t) > 0$ for all $t \in (-\epsilon, \epsilon)$ different from $0.$  Since $(a, 0)$ is generic in $\widetilde{\Sigma},$ we may assume, by shrinking $\Omega$ if needed, that for all $(x_1,\dots, x_n) \in \Omega,$ if $x_n \ne 0,$ then $\frac{\partial {f}_n}{\partial x_n}(x_1,\dots, x_n) > 0.$  If this last change of $\Omega$ makes the point $(a, 0)$ not generic in $\widetilde{\Sigma},$ we replace $(a, 0)$ by another generic point and continue to assume that $(a, 0)$ is generic in $\widetilde{\Sigma}.$

Define the definable continuous mapping $\Psi\colon \Omega \to  \mathbb R^n$ by 
$$\Psi(x_1,\dots, x_n)\ :=\ (x_1,\dots, x_{n-1}, {f}_n(x_1, \dots , x_n)).$$
Obviously, $\Psi$ is identity on $\widetilde{\Sigma}$ and differentiable on $\Omega \setminus \widetilde{\Sigma}.$ Furthermore, $\Psi$ is injective since for every $(x_1,\dots, x_{n-1}, 0) \in \widetilde{\Sigma},$ the function $x_n \mapsto {f}_n(x_1,\dots, x_n)$ is strictly increasing. In light of Lemma~\ref{Brouwer1}, $\Psi$ is a homeomorphism from $\Omega$ onto $\Psi(\Omega).$ Observe that $J\Psi(x) = \frac{\partial {f}_n}{\partial x_n}(x)$ for all $x \in \Omega \setminus \widetilde{\Sigma},$ hence $J\Psi$ is positive outside of $\widetilde{\Sigma}.$ Consequently, we can replace the mapping ${f}$ by the mapping ${f} \circ \Psi^{-1}$ without changing the sign of the Jacobian of ${f}$ on $\Omega \setminus \widetilde{\Sigma}.$ Without loss of generality, assume from now on that ${f}_n(x_1,\dots, x_n) = x_n$ on $\Omega.$

For $x = (x_1,\dots, x_n) \in \Omega \setminus \widetilde{\Sigma}$ we have
\begin{equation}\label{Eqn1}
J{f}(x)\ =\
\left|\begin{array}{ccccc}
\frac{\partial {f}_1}{\partial x_1}(x)& \cdots& \frac{\partial {f}_1}{\partial x_{n-1}}(x)&\frac{\partial {f}_1}{\partial x_n}(x)\\
\vdots&\ddots&\vdots&\vdots\\
\frac{\partial {f}_{n-1}}{\partial x_1}(x)& \cdots& \frac{\partial {f}_{n-1}}{\partial x_{n-1}}(x)&\frac{\partial {f}_{n-1}}{\partial x_n}(x)\\
0&\cdots&0&1
\end{array}\right|
\ =\
\left|\begin{array}{ccccc}
\frac{\partial {f}_1}{\partial x_1}(x)& \cdots& \frac{\partial {f}_1}{\partial x_{n-1}}(x)\\
\vdots&\ddots&\vdots\\
\frac{\partial {f}_{n-1}}{\partial x_1}(x)& \cdots& \frac{\partial {f}_{n-1}}{\partial x_{n-1}}(x)
\end{array}\right|.
\end{equation}

As ${f}$ is open, the restriction of ${f}$ to ${{f}^{-1}(\mathbb R^{n-1}\times\{0\})}$ is an open mapping from
${{f}^{-1}(\mathbb R^{n-1}\times\{0\})}$ into $\mathbb R^{n-1}\times\{0\}.$ Hence the mapping
$${g} \colon W \to \mathbb R^{n-1},  \quad x' \mapsto ({f}_1(x', 0), \dots, {f}_{n - 1}(x', 0)),$$ 
is definable, open and continuous, where $W := \{x' \in \mathbb R^{n-1}:\ (x', 0) \in \widetilde{\Sigma}\}$ is a definable open subset of $\mathbb{R}^{n - 1}.$
By \cite[Theorem~3.10]{Denkowski2017} (see also \cite[Proposition, page~298]{Gamboa1996}), the fibers of $g$ are finite.
Applying Lemma~\ref{Lemma33}(ii) to the mapping $g,$ we have that the set $R_g$ is dense in $W.$ Thus, replacing $a$ by a point in $R_g$ if necessary, we may assume that $J{g}(a)\ne 0.$

By Lemma~\ref{Lemma32}, for $1\leqslant i, j \leqslant n - 1,$ we have
$$\lim_{x \to (a, 0), \, x \in \Omega}\frac{\partial {f}_i}{\partial x_{j}}(x)\ =\ \frac{\partial {g}_i}{\partial x_{j}}(a).$$
It follows then from~\eqref{Eqn1} that
$$\lim_{x \to (a, 0), \, x \in\Omega}J{f}(x)\ =\ J{g}(a).$$
Hence, for $x \in \Omega$ close enough to $(a, 0),$ whether in ${R}_1$ or in ${R}_2,$ the sign of $J{f}(x)$ is the same as the sign of $J{g}(a).$ 
In particular, the sign of $J{f}$ is the same in ${R}_1$ and in ${R}_2.$ 

We have thus proved that the Jacobian $Jf$ has constant sign on $R_f.$ So, without loss of generality, we may assume that the Jacobian $Jf$ is positive on $R_f.$ It remains to show that
$$Jf(x) \geqslant 0 \quad \textrm{ for all } \quad x \in D_f.$$
To see this, let $x \in D_f$ be such that $Jf(x) \ne 0.$
In view of Lemma~\ref{Brouwer3}(i), there exists a definable open and bounded neighborhood $W$ of $x$ such that $\overline{W} \cap f^{-1}(f(x)) = \{x\}$ and 
\begin{eqnarray} \label{Eqn2}
\deg (f, W, f(x)) &=& \mathrm{sign} Jf(x). 
\end{eqnarray} 
On the other hand, by Lemma~\ref{Lemma33}, the set $f(R_f)$ is dense in $f(\Omega).$ Note that $f(W)$ is an open subset of $f(\Omega)$ because the mapping $f$ is open. Consequently, there exists a point $y \in f(W)$ with $\|y - f(x)\| < \mathrm{dist}(f(x), f(\partial W))$ such that 
$f^{-1}(y) \subset R_f.$ In particular, for every $w \in f^{-1}(y),$ the Jacobian matrix $df(w)$ exists and is nonsingular. It follows from Lemmas~\ref{Brouwer2}(ii) and \ref{Brouwer3}(ii) that
\begin{eqnarray}  \label{Eqn3}
\deg(f, W, f(x)) &=& \deg(f, W, y)  \ = \ \sum_{w \in f^{-1}(y)} \mathrm{sign} Jf(w) \ > \ 0.
\end{eqnarray}
Combining \eqref{Eqn2} and \eqref{Eqn3}, we get $Jf(x) > 0,$ which completes the proof of the implication  (i)~$\Rightarrow$~(ii). 

\medskip 
(i) $\Rightarrow$ (iv): 
If $n = 1,$ then ${f}$ is strictly monotone and there is nothing to prove. So for the rest of the proof we assume that $n > 1.$

By \cite[Theorem~3.10]{Denkowski2017} (see also \cite[Proposition, page~298]{Gamboa1996}), the fibers of $f$ are finite.
Hence, it suffices to show that $\dim B_f \leqslant n - 2.$ (Recall that $B_f$ denotes the set of points at which $f$ is not a local homeomorphism.)
By the inverse mapping theorem, we have $B_f \subset \Omega \setminus R_f.$ This, together with Lemmas~\ref{DimensionLemma} and \ref{DiffrentiableLemma}, implies
\begin{eqnarray*}
\dim B_f &\leqslant& \dim (\Omega \setminus R_f) \ \leqslant \ n - 1.
\end{eqnarray*}

Suppose for contradiction that $\dim B_f = n - 1.$ By analysis similar to that in the proof of the implication (i)~$\Rightarrow$~(ii), we may assume the following conditions hold:
\begin{enumerate}[{\rm (a)}]
\setcounter{enumi}{0}
\item $\Omega\setminus R_f = B_f$;
\item $B_f \subset \{x_n = 0\}$ and $f(B_f)\subset \{y_n=0\}$; 
\item $\Omega\setminus B_f$ has two connected components, denoted by $R_1$ and $R_2$ with
$${R}_1 \subset \{x_n > 0\},\ \ {R}_2 \subset \{x_n < 0\},\ \ {f}({R}_1) \subset \{y_n > 0\}\ \text{ and }\ {f}({R}_2) \subset \{y_n < 0\};$$
\item ${f}_n(x_1,\dots, x_n) = x_n$ on $\Omega;$
\item the mapping 
$${g} \colon W \to \mathbb R^{n-1},  \quad x' \mapsto ({f}_1(x', 0), \dots, {f}_{n - 1}(x', 0)),$$ 
is definable, open and continuous, where $W := \{x' \in \mathbb R^{n-1}:\ (x', 0) \in B_f\}$ is a definable open subset of $\mathbb{R}^{n - 1};$
\item $J{g}(x')\ne 0$ for all $x' \in W.$
\end{enumerate}

Let $\overline{x} := ({a}, 0) \in B_f.$ Then $f$ is not a local homeomorphism at $\overline{x},$ and so it is not injective. Hence there are sequences $x^k\to\overline x$ and $y^k\to\overline x$ such that $x^k\ne y^k$ and $ f(x^k)= f(y^k)$ for all $k.$
Taking subsequences if needed, we can suppose that the sequences $x^k$ and $y^k$ belong to only one of the sets $R_1$, $R_2$ and $B_f.$
Note that by Item~(f), the restriction of $f$ on $B_f$ is a local diffeomorphism at $\overline x.$ So $x^k, y^k\not\in B_f$ for all $k$ large enough.
With no loss of generality, assume that $x^k, y^k\in R_1.$ Furthermore, by construction, we can assume that the segment joining $x^k$ and $y^k$ is contained in $R_1$ for all $k$.

Clearly $ f$ is $C^1$ on $R_1$.
So for each $i=1,\dots,n-1$ and for each $k$, by the mean value theorem, there is a point $z^{ik}$ in the segment joining $x^k$ and $y^k$ such that
\begin{eqnarray}\label{mean}
0 & = & f_i(y^k)- f_i(x^k) \ = \ [d f_i(z^{ik})](y^k-x^k).
\end{eqnarray}
Let $$v^k:=\frac{y^k-x^k}{\|y^k-x^k\|}.$$
By Item~(d), we have $f_n(x) = x_n$ for all $x\in R_1,$ so the condition $ f(x^k)= f(y^k)$ implies $x^k_n=y^k_n$, i.e., $v^k_n=0$ for all $k.$
Furthermore, in view of~\eqref{mean}, we have 
$$[d f_i(z^{ik})](v^k)=\frac{1}{\|y^k-x^k\|}[d f_i(z^{ik})](y^k-x^k)=0.$$
Equivalently
\begin{equation}\label{sum}
\sum_{j=1}^{n-1}\frac{\partial { f}_i}{\partial x_{j}}(z^{ik})v^k_j=0.
\end{equation}

On the other hand, by Lemma~\ref{Lemma32}, for $1\leqslant i, j \leqslant n - 1,$ we have
$$\lim_{x \to ({a}, 0), \, x \in \Omega}\frac{\partial { f}_i}{\partial x_{j}}(x) = \frac{\partial {g}_i}{\partial x_{j}}({a}).$$
In addition, since $x^k\to\overline x$ and $y^k\to\overline x$, it follows that $z^{ik}\to\overline x$. Hence
$$\lim_{k\to+\infty}\frac{\partial { f}_i}{\partial x_{j}}(z^{ik}) = \frac{\partial {g}_i}{\partial x_{j}}({a}).$$
Furthermore, taking a subsequence if necessary, we can assume that the sequence $v^k$ converges to a limit $v = (v', 0).$
It follows then from~\eqref{sum} that
$$\sum_{j=1}^{n-1}\frac{\partial {g}_i}{\partial x_{j}}({a}) v_j=0.$$
Equivalently $d{g}({a})(v')=0.$
On the other hand, $d{g}({a})$ is a linear isomorphism by Item~(f).
In addition, since $v^k\to v = (v', 0)$ and $\|v^k\|=1$ for all $k$, it follows that $\|v'\|=\|v\|=1$.
These imply $d{g}({a})(v')\ne 0,$ which is a contradiction. 
Therefore the dimension of $B_f$ must be smaller than $n - 1.$

\medskip 
(iv) $\Rightarrow$ (iii): 
Let $x^0, x^1 \in R_f.$ By the inverse mapping theorem, $x^0, x^1 \in \Omega \setminus B_f.$ On the other hand, the set $\Omega \setminus B_f$ is path connected because of our assumption that $\dim B_f \leqslant n - 2.$ Hence, there exists a continuous curve 
$\alpha \colon [0, 1] \to \Omega \setminus B_f$
such that $\alpha(0) =x^0$ and $\alpha(1) = x^1.$ For each $t \in [0, 1],$ we can find an open neighbourhood $U_t$ of $\alpha(t)$ 
with $\overline{U}_t \subset \Omega$ such that the restriction of $f$ to $U_t,$ denoted by $f|_{U_t},$ is a homeomorphism. Using properties of the Brouwer degree, it is not hard to check that the function
$$[0, 1] \to \mathbb{Z}, \quad t \mapsto \deg(f|_{U_t}, U_t, f(\alpha(t))),$$
is constant. In particular, we have
$$\deg(f|_{U_0}, U_0, f(x^0)) = \deg(f|_{U_1}, U_1, f(x^1)).$$
This relation, together with Lemma~\ref{Brouwer3}(i), gives
$$\mathrm{sign} Jf(x^0) = \mathrm{sign} Jf(x^1).$$
Since $x^0, x^1$ are two arbitrary points in $R_f,$ we get the desired conclusion.
\end{proof}

Recall that a continuous mapping $f \colon \mathbb{R}^n \to \mathbb{R}^n$ is {\em piecewise affine} if there exists a set of triples $(\Omega_i, A_i, b_i), i = 1, \ldots, k,$ such that each $\Omega_i$ is a polyhedral set in $\mathbb{R}^n$ with non-empty interior, each $A_i$ is an $n\times n$-matrix, each $b_i$ is a vector in $\mathbb{R}^n,$ and
\begin{enumerate}[{\rm (a)}]
\item $\mathbb{R}^n = \cup_{i = 1}^k \Omega_i$;
\item for $i \ne j,$ $\Omega_i \cap \Omega_j$ is either empty or a proper common face of $\Omega_i$ and $\Omega_j$;
\item $f(x) = A_ix + b_i$ on $\Omega_i,$ $i = 1, \ldots, k.$
\end{enumerate}
We say that $f$ is {\em coherently oriented} if all the matrices $A_i$ have the same nonzero determinant sign. The following result is well-known; 
for more details, please refer to \cite[Theorem~2.3.1]{Scholtes2012} and the references therein.

\begin{corollary}\label{Corollary31}
Let $f \colon \mathbb{R}^n \to \mathbb{R}^n$ be a piecewise affine mapping.
Then ${f}$ is open if and only if it is coherently oriented.
\end{corollary} 
\begin{proof}
This is a direct consequence of Theorem~\ref{MainTheorem}.
\end{proof}

Finally we prove that Whyburn's conjecture is true for definable mappings. 

\begin{theorem}\label{WhyburnConjecture}
Let $f \colon \overline{\mathbb{B}^n_r} \to \overline{\mathbb{B}^n_s}$ be a definable open continuous mapping such that $f^{-1}(\mathbb{S}_s^{n - 1}) = \mathbb{S}_r^{n - 1}$ and the restriction of $f$ to $\mathbb{S}_r^{n - 1}$ is a homeomorphism. Then $f$ is a homeomorphism.
\end{theorem}

\begin{proof}
By the invariance of domain theorem (Lemma~\ref{Brouwer1}), it suffices to show that $f$ is injective. To this end, define the mapping ${g} \colon \mathbb{R}^n \to \mathbb{R}^n $ by
$${g}(x) := 
\begin{cases}
f(x) & \textrm{ if } \ \|x\| < r,\\
\frac{\|x\|}{r} f \left(\frac{r}{\|x\|} x\right) & \textrm{ otherwise}.
\end{cases}$$
Then it is easy to see that ${g}$ is a definable open continuous mapping. 
By Theorem~\ref{MainTheorem}, the fibers of ${g}$ are finite and the Jacobian $Jg$ does not change sign on $R_{g}.$ Moreover, in view of Lemma~\ref{Lemma33}(ii), $R_f$ is dense in $\mathbb{R}^n.$ Fix $r' > r.$ There exists a point $x \in R_g$ with $r < \|x\| < r'.$ By construction, $g^{-1}(g(x)) = \{x\}.$ It follows from Lemma~\ref{Brouwer3} that
$$\deg(g|_{\overline{\mathbb{B}^n_{r'}}}, \mathbb{B}^n_{r'}, g(x)) = \mathrm{sign} Jg(x) = \pm 1,$$
where $g|_{\overline{\mathbb{B}^n_{r'}}}$ stands for the restriction of $g$ to ${\overline{\mathbb{B}^n_{r'}}}.$
This, together with Lemma~\ref{Brouwer2}(iii), yields
\begin{equation} \label{PT6}
\deg(g|_{\overline{\mathbb{B}^n_{r'}}}, \mathbb{B}^n_{r'}, y) = \pm 1 \quad \textrm{ for all } \quad y \in \mathbb{B}^n_s.
\end{equation}

We now show that $f$ is injective, or equivalently, the restriction of $g$ to $\mathbb{B}_r^n$ is injective.
By contradiction, suppose that there exist two distinct points $x^0, x^1 \in \mathbb{B}_r^n$ whose image point $y \in g(\mathbb{B}_r^n) = \mathbb{B}_s^n.$ Let $U^0$ and $U^1$ be disjoint open sets containing $x^0$ and $x^1,$ respectively. Then $g(U^0) \cap g(U^1)$ is a nonempty open set and thus contains a point $y'$ of $\mathbb{B}_s^n \setminus g(\mathbb{R}^n \setminus R_g).$ This means that $g^{-1}(y')$ is a subset of $\mathbb{B}_r^n \cap R_g$ and it contains at least two points. Since the Jacobian $Jg$ does not change sign on $R_g,$ it follows from Lemma~\ref{Brouwer3} that
$$\big |\deg(g|_{\overline{\mathbb{B}^n_{r'}}}, \mathbb{B}_{r'}^n, y') \big | = \big | \sum_{x \in g^{-1}(y')} \mathrm{sign} Jg(x) \big | \geqslant 2,$$
which contradicts \eqref{PT6}. The theorem is proved.
\end{proof}

\begin{remark}{\rm
Another proof of Theorem~\ref{WhyburnConjecture} can be obtained by applying Theorem~\ref{MainTheorem} and \cite[Theorem~5.5]{Vaisala1966}; the detail is left to the reader.
}\end{remark}

\bibliographystyle{abbrv}

% \bibliography{H:/Submission/DST-BibPureMath1,H:/Submission/DST-BibAppMath1}\end{document}
%\bibliography{D:/Submission/BibPureMath1,D:/Submission/BibAppMath1}

\end{document}